\documentclass[11pt]{article}
\usepackage{proof}
\usepackage{bussproofs}
\usepackage{amssymb}
\usepackage{amsmath}
\usepackage{booktabs}
\usepackage{enumerate} 
\usepackage{xcolor}
\usepackage{setspace}
\usepackage{pstricks}
\usepackage[left=4cm,top=3cm,bottom=4cm,right=4cm]{geometry}
\usepackage{amsthm}
\usepackage{amssymb}
\theoremstyle{definition}

\newtheorem{thm}{Theorem}
\newtheorem{defn}{Definition}

\newtheorem{prop}{Proposition}

\newtheorem{exm}{Example}

\title{Most-Intersection of Countable Sets}
\author{Ahmet \c{C}evik\footnote{Gendarmerie and Coast Guard Academy, Ankara, Turkey. E-mail: a.cevik@hotmail.com}\hspace{1cm} Sel\c{c}uk Topal\footnote{Bitlis Eren University, Department of Mathematics, Bitlis, Turkey. E-mail: s.topal@beu.edu.tr.}
}
\date{}
\begin{document}
\maketitle

\begin{abstract}
We introduce a novel set-intersection operator called `most-intersection' based on the logical quantifier `most', via natural density of countable sets, to be used in determining the majority characteristic of a given countable (possibly infinite) collection of systems. The new operator determines, based on the natural density, the elements which are in `most' sets in a given collection. This notion allows one to define a majority set-membership characteristic of an infinite/finite collection with minimal information loss, compared to the standard intersection operator, when used in statistical ensembles. We also give some applications of the most-intersection operator in formal language theory and hypergraphs. The introduction of the most-intersection operator leads to a large number of applications in pure and applied mathematics some of which we leave open for further study.
\end{abstract}



\vspace{0.5cm}


One fundamental problem in mathematics and science is the specification of an average or majority dynamics of a given collection of systems. On the one hand, if the given collection is finite, a straightforward solution is to use statistical or approximation methods. The literature is full of different solutions each of which may be based on different parameters and criteria. On the other hand, the same problem for an infinite collection of systems is usually considered to be more sophisticated. Since set theory is the foundation of mathematics, and that all mathematical entities can be formally described in terms of sets, a basic way of describing a system is by using sets of objects. Hence, one can translate the problem of finding an average dynamics of a collection of systems to a problem of finding a set with an overall membership characteristic that interpolates and roughly fits the characteristic of the sets in the given collection. A trivial solution is to take the intersection of all sets in the collection so as to obtain a sub-characteristic that is common to all sets. But taking merely set-intersections causes too much information loss, particularly in the case of infinite collections. Imagine an infinite collection of sets $\{A_i\}$, for $i\in\mathbb{N}$, such that $A_i$ is defined as the set of first $i$ prime numbers. Clearly, each $A_i$ is finite. Furthermore, each prime number $p$ is to be found in infinitely many $A_j$'s and not to be found only in finitely many $A_k$'s. If we take the intersection of the collection $\{A_i\}$, we get the empty set since $A_0=\emptyset$. However, this leads to information loss when finding the majority property of the given collection. In finding the average characteristic in this case, it is reasonable to include every prime number $p$ as a part of this interpolation process since every such $p$ is a member of the `majority' of $A_i$'s.

Given a finite collection, determining whether an element is included in the majority of the sets in the collection is straightforward, as we will mention shortly. Standard known statistical methods could still be used for a finite collection of sets. However, one would also expect to extend similar approaches to infinite collections. In this paper, we are not concerned with statistical methods but we rather propose an {\em analytic} method that uses the logical quantifier `most' to define an {\em average-like} membership characteristic for a given infinite collection of sets. For this, we define a novel set-intersection operator called `most-intersection' that produces the set of elements that are in {\em most} of the sets in the given infinite collection. The operator relies on the relative natural density comparison of two infinite sets via the `most' quantifier. We cover some of the work in \cite{TopalCevik2020} to which the reader can refer for further details. Though for completeness, we will give the necessary background knowledge. This paper can be thought of as a continuation of the latter work. We also refer the reader to recent studies of Moss \cite{Moss2010}, Endrullis and Moss \cite{EndrullisMoss2019}, and a recent paper of Moss and Topal \cite{MossTopal2020}. 

\section{Most-Intersection Operator}

First we introduce a new set-intersection operator called `most-intersection', based on natural density and the logical quantifier `most', and we compare it with the standard set-theoretic intersection. We then give some applications in the next section. Our primary motivation in introducing the operator is to keep the information loss as small as possible when analytically trying to find a majority characteristic of an infinite collection.

\subsection{The Finite Case}

Although our main concern is to define and use our operator for infinite collections, let us first give the most-intersection operator for the finite case. Intuitively, the most-intersection of a finite collection of sets is interpreted in the usual way of the semantics of `most' quantifier.

\begin{defn}
The {\em most-intersection} of a finite collection $\mathcal{F}$ of sets, denoted by $\bigcap_M\mathcal{F}$, is the set of elements which are in more than half of the sets in the collection. 
\end{defn}

Consider the following simple examples for the finite case. 

\begin{exm}
Let $\mathcal{F}=\{\{1,2,3\}, \{2,3,5\}, \{4,3\}\}$ be a collection. Then, the most intersection of $\mathcal{F}$ is defined as $\bigcap_M\mathcal{F}=\{2,3\}$ since the elements $2$ and $3$ are in {\em more than half} of the sets in $\mathcal{F}$.
\end{exm}

\begin{exm}\label{exm:exm2}
Let $\mathcal{F}=\{\{a,b\}, \{a,b,c\}, \{a,c,d\}, \{a,b,d,e\}\}$ be a collection. Then, $\bigcap_M\mathcal{F}=\{a,b\}$.
\end{exm}

\begin{exm}
Suppose that $\mathcal{F}=\{\{a,b,c\}\}$. Then we have that $\bigcap_M\mathcal{F}=\{a,b,c\}$.
\end{exm}

Clearly, if $\bigcap_M \mathcal{F}=C$ and $\bigcap\mathcal{F}=D$, whenever $\mathcal{F}$ is a collection of non-empty sets, then $D\subseteq C$. Unlike the standard intersection operator, it may not be always the case that $\bigcap_M \mathcal{F}\subseteq A_i$ for all $A_i\in\mathcal{F}$. Consider the collection $\mathcal{F}=\{\{a\}, \{b\}, \{b,c\}\}$. In this case, $\bigcap_M \mathcal{F}=\{b\}$, but $\{b\}\not\subseteq \{a\}$.

It is clear to see that the most-intersection of any family with two members is same as the standard intersection. More specifically, the following proposition holds trivially.

\begin{prop}
The following statements hold for the most-intersection of two sets, where for two sets $A$ and $B$, $A\cap_M B$ means $\bigcap_M \{A, B\}$.
\begin{enumerate}
\item[(i)] $A\cap_{M} B=B\cap_{M} A$ (commutativity)
\item[(ii)] $(A\cap_{M} B)\cap_{M} C=A\cap_{M}(B\cap_{M} C)$ (associativity)
\item[(iii)] $A\cap_{M} \emptyset=\emptyset$.
\item[(iv)] $A\cap_{M} U=A\cap U=A$, where $U$ denotes the universe.
\item[(v)] $A\cap_{M} A=A\cap A=A$.
\end{enumerate}
\end{prop}
\begin{proof}
Trivial.
\end{proof}

Note that unlike the standard set intersection operator, the most-intersection of a finite collection may yield a set with cardinality greater than that of some sets in the collection.

\subsection{The Infinite Case}

Now let us consider the infinite case. For this purpose we define some notions given in \cite{TopalCevik2020}.

\begin{defn}\label{nd}
Let $ A \subseteq \mathbb{N} $ be a set and let

$$ d(A)= \lim_{n \rightarrow \infty} \dfrac{|\:A \cap \{1,2,...,n\}\:|}{n}.   $$

\noindent If the limit exists, then $d(A)$ is called the {\em lower asymptotic (natural) density} of $ A $. We will simply call this the  {\em natural density} of $A$ in the rest of the paper.
\end{defn}

So natural density is a kind of `measure' to attribute a thickness value to an (infinite) arithmetic sequence of natural numbers, such as $ d(\{k,2k,3k,4k, . . . \})= \dfrac{1}{k} $ for $k\in\mathbb{N}$. 

\begin{defn}\label{asym}
	A set $ A $ is {\em  asymptotic } to set $ B $, written $ A \sim B $, if the symmetric difference $ A\vartriangle B$ is finite. 
\end{defn}

Axioms for natural density are given by the following postulates \cite{Phdaxioms}:

\noindent Let $ d: P(\mathbb{N}) \rightarrow [0,1] $ be a function and let $ A,B \subseteq\mathbb{N} $, where $P(\mathbb{N})$ denotes the {\em power set} of $\mathbb{N}$. Then,
\begin{enumerate}[(1)]
	\item For all $A$, $ 0\leq d(A) \leq 1 $.
	\item $ d(\mathbb{N}) =1$ and $ d(\mathbb{\emptyset}) =0$.
	\item If $ A \sim B $, then $ d(A)=d(B) $.
	\item If  $ A \cap B=\emptyset $, then $ d(A)+d(B) \leq d(A\cup B) $.
	\item For all $A$ and $B$, $ d(A)+d(B) \leq 1+ d(A\cap B) $. 
\end{enumerate}

Some of the following properties, presented in the work of Grekos \cite{G1}, Buck \cite{generalizedasym} and Niven \cite{Iven}, will also be helpful for our study.
\begin{enumerate}[(i)]

	\item  $d(A)= 1- d(A^c)  $, where $A^c$ is the complement of $A$ with respect to a fixed universe.

	\item If $ A $ is a finite subset of $ \mathbb{N} $, then $ d(A)=0$.

	\item	If $ A \subseteq B $, then $ d(A) \leq d(B) $.
\end{enumerate}

The axioms are not entirely independent from each other. Note that (4) in fact follows from (2) and (5), (ii) follows from (2) and (3). Also note that (iii) follows from (4). 

Now we define a special form of the binary predicate {\em Most} in the following manner (see \cite{TopalCevik2020} for the general case). Let $U$ be countably infinite universe and let $A\subseteq U$ be a set. We say that $\textrm{Most}(U,A)$ is true iff $d(A\cap U) > d(U-A)$. In this case we say that most elements of $U$ are in $A$.

\begin{defn}\label{defn:characteristic}
Let $\{A_i\}$ be a countable collection of sets and let $w$ be an object. Define the {\em characteristic acceptance sequence} $\chi(w)$ of $w$ as follows: For a given $i\in\mathbb{N}$, if $w\in A_i$, then let $\chi_i(w)=1$; otherwise, let $\chi_i(w)=0$, where $\chi_i(w)$ denotes the $i$th element of $\chi(w)$. The {\em set interpretation} of any given characteristic acceptance sequence $\chi(w)$ is denoted by $S_{\chi(w)}$ and is defined as follows:
\begin{center}
For every $n\in\mathbb{N}$,\  $n\in S_{\chi(w)}\textrm{ iff } \chi_n(w)=1$.
\end{center}
\end{defn}

Essentially, the characteristic acceptance sequence of an object $w$ gives us an infinite 0-1 sequence such that the $i$th element of the sequence is defined as $1$ if $w$ is an element of $A_i$, and $0$ otherwise. Hence, the set interpretation of a characteristic acceptance sequence is like a projection of $\chi(w)$ onto the natural number domain.

\begin{defn}
Let $\mathcal{C}$ be a countably infinite collection of countable sets $A_i$. We define the {\em most-intersection} of $\mathcal{C}$, denoted by $\bigcap_{M}\mathcal{C}$, such that $w\in \bigcap_{M}\mathcal{C}$ if and only if
\begin{enumerate}
\item[(i)] There exists no $s$ such that $\chi_t(w)=0$ for all $t>s$,
\item[(ii)] $\textrm{Most}(\mathbb{N},S_{\chi(w)})$ is true.
\end{enumerate}
\end{defn}

There are only two cases that do not distinguish the usual intersection operator from the most-intersection operator. This is provided in the following proposition.

\begin{prop}
Let $\mathcal{F}$ be a collection of countable sets $A_i$. If $A_i=A_j$ or $A_i\cap A_j=\emptyset$ for every $i$ and $j$, then $\bigcap_{M} \mathcal{F}=\bigcap \mathcal{F}$.
\end{prop}
\begin{proof}
If $A_i=A_j$ for every $i$ and $j$, then clearly $\bigcap_{M}\mathcal{F}=A_i=\bigcap\mathcal{F}$ for any $i$. If $A_i$ and $A_j$ are disjoint for all $i$ and $j$, then since there is no element $x$ contained in two distinct sets, {\em a fortiori} there is no such $x$ which is contained in most sets in $\mathcal{F}$. Then, $\bigcap_{M}\mathcal{F}=\bigcap\mathcal{F}=\emptyset$.
\end{proof}

Another observation about the most-intersection operator concerns distribution rules and closures. In the next proposition we omit for readability writing the index $i$ belonging to the set-theoretic operators.

\begin{prop} Let $U$ be a countable universe, $B\subseteq U$ be a set.
\begin{enumerate}
\item[(i)] For any countable collection $\{A_i\}$ of subsets of $U$, $B\cup \bigcap_M A_i=\bigcap_M (B\cup A_i)$.
\item[(ii)] It is not the case that $U-\bigcup A_i = \bigcap_M (U-A_i)$ for every countable $\{A_i\}$.
\item[(iii)] It is not the case that $U-\bigcap_M A_i = \bigcup (U-A_i)$ for every countably infinite $\{A_i\}$.
\end{enumerate}
\begin{proof}
For (i), clearly $B$ is added into every $A_i$. Hence, by definition, elements of $B$ are in most of $B\cup A_i$. For (ii), as a counter-example to the equality, consider the case that we are given a collection $\{A,B,C\}$ such that $B\subseteq C$ and that $A$ is disjoint with both $B$ and $C$. To prove (iii), suppose for a contradiction that $x$ is in most $A_i$ but there exists some $j$ such that $x\not\in A_j$. Then, $x\in\bigcap_M A_i$. Hence, $x\not\in U-\bigcap_M A_i$. But then $x\in (U-A_j)$. A contradiction.
\end{proof}
\end{prop}

Note that although (i) would remain true if the most-intersection was replaced with the standard intersection operator, the equalities given (ii) and (iii) are true for standard intersection operator given any family $\{A_i\}$.
\vspace{0.5cm}

We would also like to know that to what extent two given countable sets are similar to each other in terms of the `most' relation with respect to the set of natural numbers. For this we define a similarity predicate based on the concept of natural density.


\begin{defn}\label{defn:similar}
Let $L_1$ and $L_2$ be two countable subsets of $\mathbb{N}$. We say that $\textrm{MostSim}(L_1, L_2)$ holds if and only if $\textrm{Most}(\mathbb{N}, L_1)\leftrightarrow \textrm{Most}(\mathbb{N},  L_2)$. In this case, we say that $L_1$ and $L_2$ are {\em mostly similar}.
\end{defn}


\begin{thm}
$\textrm{MostSim}(X,Y)$ is an equivalence relation. 
\end{thm}
\begin{proof}
The proof immediately follows from the reflexivity, symmetry and transitivity of the biconditional connective in the definition. 
\end{proof}


This property gives one the opportunity to form a hierarchy of equivalence classes of sets based on the MostSim relation. For example, in case of finite automata that we discuss in the next section, one could classify a given collection of finite automata according to their `similarities' with respect to each other. For instance, if $\{A_i\}$ is a collection of finite automata and that $A_k$ is mostly similar to $A_l$, then we may put the language of $A_k$ and the language of $A_l$ in the same density level of the MostSim hierarchy. We leave the investigation concerning the equivalence classes formed by MostSim for future research as this deserves a separate study.

\section{Applications of Most-Intersection}
One of the many applications of the most-intersection operator concerns the determination of the overall characteristic of {\em infinitely many} dynamical systems. This in fact amounts to perform statistics over an infinite data. Given infinitely many states of a dynamical system, say, we can {\em analytically} find the majority state with the most-intersection operator. For simplicity we shall keep the examples as abstract and simplified as possible. For example, we may consider primitive computing models and describe their majority characteristic based on the introduced notion of density via the `most' quantifier. 

\subsection{Formal Languages}

The model of computation we take will be a finite state machine (or finite automaton) for convenience. Given a countably infinite collection of finite automata, we want to determine the majority characteristic of the given collection using natural density via `most' quantifier.

We adopt the usual definition of deterministic finite automata. A {\em finite automaton} is a 5-tuple $A=(Q,\Sigma,\delta,q_0,F)$, where $Q$ is a finite set of states, $\Sigma$ is a finite set of symbols, $\delta:Q\times\Sigma\rightarrow Q$ is the transition function, $q_0\in Q$ is the initial state and $F\subseteq Q$ is the set of accepting states.

An {\em alphabet} is a finite set of symbols. A {\em string} is a finite sequence of symbols over an alphabet. A {\em language} is a set of strings. The {\em length} of a string $w$, denoted by $|w|$, is the number of symbols in $w$. The unique string of length 0 is called the {\em empty string} and is denoted by $\epsilon$. The {\em concatenation} of two strings $u$ and $v$ is simply denoted by $uv$. Note that for any string $u$, we have that $\epsilon u=u\epsilon=u$. For a set of symbols $\Sigma$ and a given $i\in\mathbb{N}$, $\Sigma^i$ denotes the set of all strings of length $i$ over $\Sigma$.

A finite automaton $A$ {\em accepts} a string $w=a_1a_2\ldots a_n$ if $(q,a_n)=q_a$ for some $q\in Q$ and $q_a\in F$; otherwise say that $A$ {\em rejects} $w$. A language $L$ is called {\em regular} if there exists a finite automaton which accepts exactly all strings in $L$. The language of a finite automaton $A$ is denoted by $L(A)$ and it is the set of all strings accepted by $A$, in which case we will say that $A$ {\em recognizes} $L$.

\vspace{0.5cm}

Regular languages can also be expressed in terms of algebraic expressions. 






Given an alphabet $\Sigma$, the following expressions are regular expressions:

Basis step:
$\emptyset$ denoting the set $\emptyset$.
$\epsilon$ denoting the set containing only the empty string.
Any symbol $a$ belonging to $\Sigma$ denoting the set containing only the symbol $a$.

Induction step:
If $R$ and $S$ are regular expressions, then so are the following expressions:

$RS$ denoting the set of strings that can be obtained by concatenating a string in $R$ and a string in $S$. 

$R + S$ denoting the union of sets described by $R$ and $S$. 

$R^*$ denoting the smallest superset of the set described by $R$ that contains $\epsilon$ and is closed under string concatenation. 
\vspace{0.5cm}

Regular expressions and finite automata are equivalent forms of describing regular languages. That is, a language is regular if and only if there exists a finite automaton that recognizes it if and only if there exists a regular expression that generates it. Note that regular languages are closed under finite unions. That is, given any two regular languages $L_1$ and $L_2$, $L_1\cup L_2$ is a regular language. However, infinite union of regular languages may not be regular.\footnote{Proving that a language is {\em not} a regular language requires what is known as the {\em pumping lemma}}. To see this, for each $i\in\mathbb{N}$, let $L_i=0^i 1^i$. Now each $L_i$ is a regular language but is it a known fact that
\[
L=\bigcup_i L_i=\{0^i 1^i: i\in\mathbb{N}\}
\]
is not a regular language. Similar argument can be used to show that regular languages are closed under finite intersections but not under infinite intersections.
\vspace{0.5cm}

In order to begin studying the average behavior of a given collection of finite automata, we begin with the simplest approach: Taking the intersection of strings accepted by every automaton in the collection.

\begin{defn}
The {\em intersection language} of a given class $\{A_i\}$ of finite automata is the set of strings $w$ such that $w$ is accepted by every $A_i$.
\end{defn}

When determining the intersection of strings accepted by a collection of finite automata, there may be a string accepted by the majority of finite automata yet not accepted by the minority. Hence, whenever we want to characterize the `average' behavior of the given collection of finite state machines, a significant amount of information is omitted if we just take the intersection of languages recognized by them. As a consequence, this approach may not always give a reasonable characterization of the given class of finite automata. For a better approximation in determining their overall characteristic, we propose a new method which uses natural density and the quantifier `most'. The advantage of this method is that we take into consideration the strings accepted by the majority of finite automata. Naturally, the output language is a weaker form of the intersection language, yet with a greater coverage and spectrum, and certainly with less information loss.

Given a collection of finite automata, we want to find a set of strings, call it the {\em density language}, whose members are accepted by the `majority' of finite automata. This will simply be defined as the most-intersection of languages of the machines in the given collection. Defining the intersection of all strings which are accepted by `all' finite automata may lead to undesired and inefficient results, as well as information loss. It may very well be the case that there are strings which are accepted by all but finitely many automata in the given collection or even rejected by infinitely many automata in a scarce manner. It is then reasonable to include these strings in the density language since a finite set is not `dense' in any infinite set. Having a finite number of exceptions is not the only case where we want to include them in the density language. It may also be the case that some strings happen to be rejected by the same ratio as the natural density of prime numbers. Set of finite automata `indexed' by primes is still not dense in a given infinite collection of finite automata `indexed' by natural numbers.
\vspace{0.5cm}

We assume there is a uniform enumeration between natural numbers and strings, e.g. strings can be ordered with respect to the lexicographical order. Let $\{L_i\}$ be a collection of languages and let $w$ be a string. We define the {\em characteristic acceptance sequence} $\chi(w)$ of a string $w$ the same way as in Definition \ref{defn:characteristic}.

\begin{defn}
Let $\mathcal{C}$ be a collection of finite automata and let $L_i$ be the languages recognized by each finite automaton $A_i$ in $\mathcal{C}$. We define the {\em density language} of $\mathcal{C}$, denoted by $L_d(\mathcal{C})$, to be $\bigcap_M \{L_i\}$.
\end{defn}

We ask if $L_d({\mathcal{C}})$ always defines a regular language. This answer turns out to be negative.

\begin{thm}
There exists a countable collection $\mathcal{C}$ of regular languages such that $L_d(\mathcal{C})$ is non-regular.
\end{thm}
\begin{proof}
We define a countable collection $\mathcal{C}$ of regular languages. Define each $L_n$ by induction on $n$. Let $L_0\in\mathcal{C}$ be the set $\{01\}$. Define $L_{n+1}=L_n\cup\{0^{n+1}1^{n+1}\}$. Now each $L_i$ is finite and so it is regular. Given a fixed index $i$, since every $L_j$ is upward closed for $j>i$, there are only finitely many languages in $\mathcal{C}$ that $0^i1^i$ is not in $L_i$. Thus, $L_d(\mathcal{C})=\bigcup_i L_i$. So the density language $L_d(\mathcal{C})$ is then equal to the set $\{0^n1^n : n\in\mathbb{N}\}$ which is known to be non-regular.
\end{proof}

\subsection{Graph Systems}

An important application of the most-intersection operator turns out to be related with the problem of determining the average behavior of a discrete system, whether the system is of finite or infinite size. We shall consider discrete evolutionary systems, meaning that the evolution is carried in discrete stages. By evolution we mean any system which can be represented by graphs and the neighborhood relation. So what we mean by a system is really a graph having some neighborhood relations among its vertices. But what about the representation of these systems? An evolutionary system of this kind can be thought to be given in the form of a hypergraph. Hypergraph is a generalization of a standard graph except that the edge relation is not restricted to two vertices like in standard graphs where the edge relation is strictly between the two vertices of the graph. For a detailed account on hypergraphs, we refer the reader to Bretto \cite{hypergraphs}. However let us now give some definitions that we need for our work.

\begin{defn}
A {\em hypergraph} is a pair $H=(V,E)$, where $V$ is a set of vertices and $E$ is a non-empty subset of the power set of $V$. Elements of $E$ are called {\em hyperedges}. 

The {\em order} of the hypergraph $H = (V,E)$ is defined as the cardinality of $V$, i.e. $|V|$, and the {\em size} is defined as the cardinality of $E$, i.e. $|E|$.
\end{defn}

\begin{exm}\label{exm:hypergraph}
Consider a hypergraph $H=(V,E)$, where $V=\{v_1, v_2, v_3, v_4, v_5, v_6\}$ and $E=\{e_1, e_2, e_3, e_4, e_5\} = \{ \{v_1,v_4\}, \{v_4,v_5\}, \{v_1,v_2,v_3\}, \{v_2,v_3,v_6\}, \{v_3,v_4,v_6\} \}$. Now $V$ has order 6 and size 5. Note that the hyperedges here do not just connect two vertices but they connect several vertices. So the (hyper)edge relation is in fact an $n$-ary relation for any $n\geq 0$, provided that $n \leq |V|$.
\end{exm}

For our purpose we consider countable hypergraphs, i.e., those which have countably many vertices but also countably many hyperedges. Let us first examine how to determine the ``average" behavior of a finite system given in terms of a finite hypergraph. This fits the most cases in real life applications. Assume that we are given a finite hypergraph containing finitely many vertices. For convenience we can take Example \ref{exm:hypergraph}. We can imagine the hyperedges of $H$ as the evolution of the system that the graph represents. For instance, the system $H$ has 5 hyperedges, and we may view this as that the system evolves from $e_1$ up to $e_5$ in discrete stages. That is, the system starts with the state $e_1$ and evolves to state $e_2$, and in the next stage it evolves to $e_3$, then to $e_4$, and then finally to $e_5$. In fact, since we want to take their ``average" in which order we take the evolution does not really matter for the analysis. The important point here is that the vertices of the given hypergraph can be thought of as a set of entities that can be in relation to one another and can change at every next stage. This defines a discrete evolutionary system. So now let us consider $H$ as given above. We ask what the {\em average state} of this system is. We want to draw a majority state so to speak. Of course this is easily done by straightforward statistical methods considering that we are working with a finite graph. But we want to compute this entirely using the logical language, i.e., using the {\em most} semantics of quantifying logic. According to the semantics of {\em most}, the majority state will be determined by considering the set of vertices which is related with the neighborhood relation with the ``most" of vertices in the hypergraph. In the case of our example above, $v_3$ appears in most hyperedges in $E$. So then we want $v_3$ to be included in the majority state of the system $H$. But $v_4$ also appears in most hyperedges. Then we also want $v_4$ to be added into the majority state. The majority state is then defined by the most-intersection of $E$, which gives us $\{v_3, v_4\}$. Then, $\bigcap_M E$ defines a new hyperedge, called the {\em average state} of $H$. Note that the average state need not be an element of $E$. Nevertheless, it gives the average state analytically determined by the logical quantifier most. Furthermore, in some cases $\bigcap_M E$ may just be empty. In this case we would call the hypergraph {\em balanced}. If a hypergraph is balanced, this means that the system it represents evolves rather homogenously over the vertices. 

Our intention of finding the majority state of a given system should now be clear for the finite case. It is in fact quite natural to apply the most-intersection operator to obtain this information. But now let us consider an infinite system represented by a countably infinite hypergraph, i.e., a hypergraph with countably infinite vertices and countably infinite hyperedges. If we want to find the average state, i.e., average hyperedge, of this kind of system, we see that the usual finitistic statistical methods will fail to provide us the desired information. We shall then use natural density and apply the most-intersection operator to compute the set of vertices which are adjacent, in the hyper sense, with the ``most" of the vertices in the given hypergraph. Given an infinite hypergraph, fix an enumeration of its vertices $v_1, v_2,\ldots$. The set $E$ of hyperedges will be a countable subset of the power set of the set of vertices. Now all we need to do is take the most intersection of $E$. Since the most-intersection gives us all $v_i$'s such that $v_i$ is in most $e\in E$, the intersection gives us a single hyperedge which corresponds to the average hyperedge of the hypergraph in the ``most" quantifier sense.

Computing the average characteristic of an evolutionary discrete system by this method gives a novel way of analytically determining in what states the system ``mostly" appears. Most real life application falls into the finite case category. Computing such characteristics within the logical framework would also allow one to implement most-intersection in logic programming and knowledge base systems. For example, given a set of objects and a set of properties, one could do an inference over these objects---solely by using the logical quantifier {\em most}---to find which objects satisfy most of the given properties. This of course can be generalised to countable sets whenever possible.

\section{Conclusion}

In this work, we introduced a new set-theoretical intersection operator based on the quantifier ``most" and natural density. We intended to use the operator for analytically finding a majority characteristic of a given system. On a pure mathematical level, this allows us to define a majority membership characteristic of a given infinite family of sets. We also defined a similarity relation, called MostSim, based on the ``most" semantics, which turned out to be an equivalence relation. This in turn allows one to study the possible equivalence class hierarchy formed by the MostSim relation and observe the algebraic structure admitted by the similarity of subsets of natural numbers under the ``most" quantifier. We also gave a short application of the most-intersection operator on formal languages and hypergraphs to emphasise on computing a majority characteristic using natural density in the ``most" quantifier sense. Consequently, applications of most-intersection for computing averages and majorities of given systems are abundant. We leave other applications for future study.

\end{document}